\documentclass[11pt,psamsfonts]{amsart}
\usepackage{amsmath}
\usepackage{amsthm}
\usepackage{amssymb}
\usepackage{amscd}
\usepackage{amsfonts}
\usepackage{amsbsy}
\usepackage{epsfig,afterpage}
\usepackage[dvips]{psfrag}
\usepackage{color}     
\usepackage{overpic}
\usepackage{graphicx}

\newcommand{\R}{\ensuremath{\mathbb{R}}}

\newcommand{\N}{\ensuremath{\mathbb{N}}}

\newtheorem {theorem} {Theorem} 
\newtheorem {proposition} [theorem] {Proposition}

\newtheorem {lemma} [theorem] {Lemma}
\newtheorem {definition} [theorem] {Definition}
\newtheorem {remark} {Remark}

\begin{document}

\title[Attractivity, degeneracy and codimension of PSVFs]{Attractivity, degeneracy and codimension of a typical singularity in 3D piecewise smooth vector fields}

\author[T. Carvalho and M.A. Teixeira] {Tiago  Carvalho$^1$ and Marco A. Teixeira$^2$}

\address{$^1$ FC--UNESP, CEP 17033--360, Bauru, S\~ao Paulo, Brazil}

\address{$^2$ IMECC--UNICAMP, CEP 13081--970, Campinas, S\~ao Paulo, Brazil and UFSCar-campus Sorocaba CEP 18052-780, Sorocaba, S\~ao Paulo, Brazil}

\email{tcarvalho@fc.unesp.br}

\email{teixeira@ime.unicamp.br}


\date{}
\dedicatory{} \maketitle

%

\begin{abstract}
We address the problem of understanding the dynamics around typical singular points of $3D$ piecewise smooth vector fields. A model $Z_0$ in $3D$ presenting a  T-singularity is considered and a complete picture of its dynamics is obtained in the following way: \textit{(i)} $Z_0$ has an invariant plane $\pi_0$ filled up with periodic orbits (this means that the restriction $Z_0 |_{\pi_0}$ is a center around the singularity), \textit{(ii)} All trajectories of $Z_0$ converge to the surface $\pi_0$, and such attraction occurs in a very non-usual and amazing way, \textit{(iii)} given an arbitrary integer $k\geq 0$ then $Z_0$ can be approximated by $\pi_0$-invariant piecewise smooth vector fields $Z_{\varepsilon}$ such that  the restriction  $Z_{\varepsilon} |_{\pi_0}$ has exactly $k$-hyperbolic limit cycles, \textit{(iv)} the origin can be chosen as an asymptotic stable equilibrium of $Z_{\varepsilon}$ when $k=0$, and finally, \textit{(v)} $Z_0$ has infinite codimension in the set of all $3D$ piecewise smooth vector fields.
\end{abstract}

\section{Introduction}\label{secao introducao}

Our interest in this paper is to study some qualitative aspects of \textit{piecewise smooth vector fields} (PSVFs for short) in $3D$.  Roughly speaking, as stated in \cite{Tiago-Claudio-Marco-CodimensaoFoldFold},  a PSVF $Z=(X,Y)$  is a pair of $C^r$-vector fields $X$ and $Y$ (both defined on $\R^3$), in such a way that just their restrictions to some regions (half-spaces) separated by a  codimension one surface  $\Sigma$ (called \textit{switching manifold}) are considered. In Section \ref{secao preliminares} we give a precise definition. 

In this context, some of the points with richer dynamics are those ones where the trajectories of $X$ and/or $Y$ are tangent to $\Sigma$. These points are called \textit{tangential singularities}.
The most known tangential singularity of a smooth system $X$ is the \textit{fold singularity} (also called \textit{fold point}), which is characterized by the quadratic contact of an orbit of $X$ with $\Sigma$. %
 A fold point $p$ can be visible or invisible. It is \textit{visible} for $X$ if the $X$-trajectory passing though $p$ remains in the same side where $X$ is defined, otherwise it is \textit{invisible}. In $3D$, generically there exists a curve of tangential singularities $S_X \subset \Sigma$ passing through $p$ (the same for $Y$). 

If $p \in \Sigma$ is a fold singularity of both systems $X$ and $Y$ then it is called a \textit{two-fold singularity}. 
This singularity is a prototypical  model  in the generic  classification of singularities in  PSVFs.  As pointed out in  \cite{Jeffrey-colombo-2011}, a two-fold singularity is an important organizing centre because it brings together all of the basic forms of dynamics possible in a PSVF. There are many distinct  topological types  of two-fold singularities  and the most interesting of them is the so called \textit{T-singularity}.  We say that $p$ is a T-singularity (or \textit{Teixeira-singularity}  $-$ due to the pioneering work \cite{T1} $-$  or \textit{invisible two-fold singularity}) for $Z=(X,Y)$ if $p$ is an invisible fold point of both $X$ and $Y$ and $S_X$ meets $S_Y$ transversally at $p$ (see Figure \ref{fig t sing}). The interested reader can see more details about the T-singularity in \cite{diBernardo-electrical-systems,Jeffrey-T-sing,Jeffrey-colombo,Jeffrey-colombo-2011,Rony,J-T-T1,J-T-T2,T1}.

\begin{figure}[!h]
	\begin{center}\psfrag{M}{$\Sigma$}\psfrag{Sx}{$S_X$}\psfrag{Sy}{$S_Y$} \epsfxsize=6cm
		\epsfbox{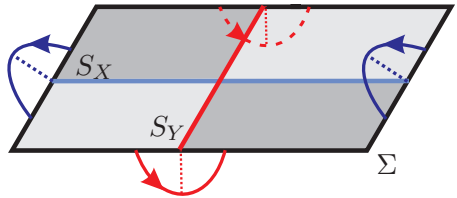} \caption{\small{T-singularity.}}
		\label{fig t sing}
	\end{center}
\end{figure}

Our primary concern in this article is to show the existence of a infinite codimension two-fold singularity. In order to do it,  we study smooth nonlinear perturbations of a specific model of $3D$ degenerate T-singularity  and a complete picture of its amazing dynamics is exhibited. It is worth noting that some interesting bifurcations of such T-singularity  are exhibited.

A formal codimension study of a singularity must contemplate the minimal number of parameters necessary in order to obtain all topological types of dynamical systems around the dynamical system presenting the singularity. In \cite{Tiago-Claudio-Marco-CodimensaoFoldFold} the authors perform this study for a planar two-fold singularity of PSVFs.

Some ideas and constructions presented in \cite{Tiago-Claudio-Marco-CodimensaoFoldFold} and \cite{Jeffrey-colombo}  are adapted in our approach. As consequence, we are able to produce an arbitrary number of topological types of PSVFs in a small neighborhood of a PSVF presenting a T-singularity. So, this T-singularity has infinite codimension.

Regardless of, this work fits into a general program for understanding the dynamics of higher dimension vector fields expressed by:
\begin{equation}\label{eq relay}
\dot{u} = F(u) + sgn(f(u)) K
\end{equation}
where $u=(u_1,u_2, \hdots, u_n) \in \R^n$, $K$ is a constant $n$-dimensional vector, $f: \R^n \rightarrow \R$  is a smooth mapping and $F: \R^n \rightarrow \R^n$, in general, is not smooth in the switching manifold $\Sigma = \{ f^{-1}(0)  \}$.

Systems like \eqref{eq relay} are wide used in applications and appears in models of Mechanics Engineering (see \cite{Brogliato,Dixon,Leine}), Electric Engineering (see \cite{diBernardo-electrical-systems,Kousaka}), Biologic/Control Theory/Economics (with sudden external influences, see \cite{Rossa}), among others. In fact, every system susceptible to \textit{on-off} operations are modelled by systems like \eqref{eq relay} which imposes an interdisciplinary aspect on this theory.

The \textit{sgn function} produces, in general, lost of differentiability  of \eqref{eq relay} when the trajectories pass throught  $\Sigma$. So, \eqref{eq relay} is a  non-smooth (or piecewise smooth) system.  As also happens for smooth systems, some choices on the functions at the right-side of \eqref{eq relay} generate non stable behavior. There are in the current literature a huge variety of papers dealing with stability conditions  of models like \eqref{eq relay} (see \cite{Teixeira-JDE-1977,T1}   among others).

It is worth mentioning that in \cite{Anosov-stability} Anosov proves the asymptotic stability of the origin of \eqref{eq relay} when $F$ is a linear vector field. After, K\"{u}pper-Kunze-Hosham found invariant varieties of \eqref{eq relay} when $K=0$. In fact, they show the existence of  invariant varieties (cones) and prove that the  trajectories of \eqref{eq relay} twist (curl up) until it (see \cite{Kunze-1998,Kupper-2008,Kupper-Hosham-2011}).

In our  approach the following $3D$ system is considered: 
\[ Z_{0}(x,y,z)=(\dot{x},\dot{y},\dot{z})  = F(x,y,z) + sgn(z) K,\]
where $f(x,y,z)=z$, \[F(x,y,z)=\frac{1}{2} \displaystyle\Biggl( \begin{tabular}{ccc}
-1  & -1  & 0 \\ 
-1  & -1  & 0 \\ 
1  &  -1 & 0 \\ 
\end{tabular} \Biggl) \Biggl(\begin{tabular}{c}
$x$ \\ 
$y$ \\ 
$z$ \\ 
\end{tabular} \Biggl) + \frac{1}{2} sgn(z) \Biggl( \begin{tabular}{ccc}
-1  & -1  & 0 \\ 
-1  & -1  & 0 \\ 
-1  &  -1 & 0 \\ 
\end{tabular} \Biggl) \Biggl( \begin{tabular}{c}
$x$ \\ 
$y$ \\ 
$z$ \\ 
\end{tabular} \Biggl)\]and \[K=   \Biggl( \begin{tabular}{c}
-2 \\ 
2 \\ 
0 \\ 
\end{tabular} \Biggl), \]
or equivalently:

\begin{equation}\label{eq Z sem F}
 Z_{0}(x,y,z) = \left\{
      \begin{array}{ll}  X(x,y,z)=(-1-(x+y),1-(x+y),-y)    & \hbox{if $z \geq 0$,} \\
        Y_{0}(x,y,z) = (1,-1, x)                                        & \hbox{if $z \leq 0$.}
      \end{array}    \right.
\end{equation}


%

%

In fact, as exhibited in \cite{Kupper-2008,Kupper-Hosham-2011}, we also identify an invariant manifold in our model. Morevoer, the trajectories converge to it in a very unusual way. The richness of our model is revealed after a suitable perturbation of it. In such case we conserve the invariant manifold and identify the birth of an arbitrary number of limit cycles. Also, we obtain the asymptotic stability of the origin (as  in \cite{Anosov-stability}).


Now we state the main results of the paper.

\vspace{.3cm}

\noindent {\bf Theorem A.}
{\it Let $Z_0$ be given by \eqref{eq Z sem F} and $\pi_0$ the plane $\{ y+x=0 \}$. For each 
integer $k \geq 0$, there exists a one-parameter family of $\pi_{0}$-invariant PSVFs
$Z_{\varepsilon}$ satisfying:}
\begin{itemize}
\item[(a)] $Z_{\varepsilon} \rightarrow Z_0$ when $\varepsilon \rightarrow 0$; 

\item[(b)] $Z_{\varepsilon}$ has exactly $k$ hyperbolic limit cycles in a neighborhood of the origin. The same holds for $k=\infty$
and,

\item[(c)] All trajectories of $Z_0$ and
$Z_{\varepsilon}$ converge to $\pi_0$.

\item[(d)] When $\varepsilon > 0$ and $k=0$, the origin is an asymptotic stable equilibrium point for $Z_{\varepsilon}$.
\end{itemize}

\vspace{.3cm}

An immediate consequence of Theorem A is: 

\vspace{.3cm}

\noindent {\bf Theorem B.}
{\it \label{corolario centro cod infinita}
The PSVF $Z_0$, given by \eqref{eq Z sem F}, has infinite codimension.}

\vspace{.3cm}

%
%
%
%

%
%
%
%
%
%
%

The paper is organized as follows. In Section \ref{secao preliminares} we introduce the terminology, some definitions and the basic theory about PSVFs.
In Section \ref{secao propriedades sistema} we present properties
towards the understanding of the phase portrait of \eqref{eq Z sem
F}. In Section \ref{secao propriedades sistema perturbado} suitable
perturbations of \eqref{eq Z sem F} are considered and the birth of
limit cycles are explicitly exhibited. In Section \ref{secao prova
resultados} we prove the main results. In Section \ref{secao
conclusão} we made a brief conclusion about the results in this
paper and in Section \ref{secao apendice} we picture some numerical
analysis and the phase portrait of  \eqref{eq Z sem F} around the
origin.

\section{Preliminaries}\label{secao preliminares}

In this section we give a brief review of the theory.

Let $V$ be an arbitrarily small neighborhood of $0\in\R^3$. We
consider a codimension one manifold $\Sigma$ of $\R^3$ given by
$\Sigma =f^{-1}(0),$ where $f:V\rightarrow \R$ is a smooth function
having $0\in \R$ as a regular value (i.e. $\nabla f(p)\neq 0$, for
any $p\in f^{-1}({0}))$. We call $\Sigma$ the \textit{switching
manifold} that is the separating boundary of the regions
$\Sigma^+=\{q\in V \, | \, f(q) \geq 0\}$ and $\Sigma^-=\{q \in V \,
| \, f(q)\leq 0\}$. Throughout  paper we assume that $\Sigma
=f^{-1}(0),$ where $f(x,y,z)=z.$

Designate by $\chi$ the space of C$^r$-vector fields on
$V\subset\R^3$ endowed with the C$^r$-topology, with $r \geq 1$
large enough for our purposes. Call $\Omega^r$ the space of vector
fields $Z: V \rightarrow \R ^{3}$ such that
\begin{equation}\label{eq Z}
 Z(x,y,z)=\left\{\begin{array}{l} X(x,y,z),\quad $for$ \quad (x,y,z) \in
\Sigma^+,\\ Y(x,y,z),\quad $for$ \quad (x,y,z) \in \Sigma^-,
\end{array}\right.
\end{equation}
where $X=(X_1,X_2,X_3) , Y = (Y_1,Y_2,Y_3) \in \chi$. We endow
$\Omega^r$ with the product topology. The trajectories of $Z$ are
solutions of  ${\dot q}=Z(q)$ and we will accept it to be
multi-valued in points of $\Sigma$. The basic results of
differential equations, in this context, were stated in \cite{Fi}.

On $\Sigma$ we distinguish the following regions:

$\bullet$ Crossing Region: $\Sigma^c=\{ p \in \Sigma \, | \, X_3(p)\cdot Y_3(p)> 0 \}$.
Moreover, we denote $\Sigma^{c+}= \{ p \in \Sigma \, | \,
X_3(p)>0,Y_3(p)>0 \}$ and $\Sigma^{c-} = \{ p \in \Sigma \, | \,
X_3(p)<0,Y_3(p)<0 \}$.

$\bullet$ Sliding Region: $\Sigma^{s}= \{ p \in \Sigma \, | \, X_3(p)<0,
Y_3(p)>0 \}$.

$\bullet$ Escaping Region: $\Sigma^{e}= \{ p \in \Sigma \, | \, X_3(p)>0
,Y_3(p)<0\}$.

When $q \in \Sigma^s \cup \Sigma^e$, following the Filippov's convention, the \textbf{sliding vector field}
associated to $Z\in \Omega^r$ is the vector field  $\widehat{Z}^s$ tangent to
$\Sigma^s \cup \Sigma^e$ and expressed in coordinates as
\begin{equation}\label{eq campo filippov}\widehat{Z}^s(q)= \frac{1}{(Y_3 - X_3)(q)} ((X_1 Y_3 - Y_1 X_3)(q),(X_2 Y_3 - Y_2 X_3)(q),0).\end{equation}
Associated to \eqref{eq campo filippov} there exists the planar \textbf{normalized sliding vector field}
\begin{equation}\label{equacao campo normalizado}Z^{s}(q)=((X_1 Y_3 - Y_1 X_3)(q),(X_2 Y_3 - Y_2 X_3)(q)).
\end{equation}

\begin{remark}\label{obs campo normalizado topol equivalente}Note that, if
$q\in \Sigma^s$ then $X_3(q)<0$ and $Y_3(q)>0$. So, $(Y_3-X_3)(q)>0$ and therefore,
 $\widehat{Z}^{s}$ and $Z^{s}$ are topologically equivalent in $\Sigma^s$ since they have the same orientation and can be C$^r$-extended to
 the closure $\overline{\Sigma^s}$ of $\Sigma^s$. If $q\in \Sigma^e$ then $\widehat{Z}^{s}$ and $Z^{s}$ have opposite orientation.
\end{remark}

In this context, a rich dynamics occurs on those points  $p \in \Sigma$ such
that $X_3(p)\cdot Y_3(p) =0$, called \textbf{tangential
singularities of $\mathbf{Z}$}  (i.e., the trajectory through $p$ is
tangent to $\Sigma$).



For practical purposes, the contact between the smooth vector field
$X$ and the switching manifold $\Sigma = f^{-1}(0)$ is characterized
by the expression $X.f(p)=\left\langle \nabla f(p),
X(p)\right\rangle = 0$  where $\langle . , . \rangle$ is the usual
inner product in $\R^3$. In this way, we say that a point $p \in
\Sigma$ is a \textit{fold point} of $X$ if $X.f(p)=0$ but
$X^{2}.f(p)\neq0$, where  $X^i.f(p)=\left\langle \nabla X^{i-1}.
f(p), X(p)\right\rangle$ for $i\geq 2$. Moreover, $p\in\Sigma$ is a
\textit{visible} (respectively {\it invisible}) fold point of $X$ if
$X.f(p)=0$ and $X^{2}.f(p)> 0$ (respectively $X^{2}.f(p)< 0$). In
addition, a tangential singularity $q$ is \textit{singular} if $q$
is a invisible tangency for both $X$ and $Y$.  On the other hand, a
tangential singularity $q$ is \textit{regular} if it is not
singular.

 Call
$S_X$ (resp. $S_Y$) the set of all tangential singularities of $X$
(resp. $Y$).  In $3D$, a point of $\Sigma$ at which two curves of
fold singularities $S_X$ and $S_Y$ meet is called a \textit{two-fold
singularity}.
  This singularity is a prototypical  model  in the generic  classification of singularities in  PSVFs.
   As pointed out in  \cite{Jeffrey-colombo-2011}, a two-fold singularity is an important organizing centre because it brings together all of the basic
   forms of dynamics possible in a PSVF. There are many distinct  topological types  of two-fold singularities  and the most interesting of them is the
    so called \textit{T-singularity}.  We say that $p$ is a T-singularity (or \textit{Teixeira-singularity}  or \textit{invisible two-fold singularity}) for $Z=(X,Y)$ if $p$ is an invisible fold point of both $X$ and $Y$ and the intersection of $S_X$ and $S_Y$ is transversal at $p$ (see Figure \ref{fig t sing}).
It is easy to check that in the model \eqref{eq Z sem F} the origin
is a T-singularity. In this article, we study smooth nonlinear
perturbations of this model and a complete picture of its dynamics
is exhibited. The interested reader can see more details about the
T-singularity in
\cite{diBernardo-electrical-systems,Jeffrey-T-sing,Jeffrey-colombo,Jeffrey-colombo-2011,Rony,J-T-T1,J-T-T2,T1}.

It is worth to say that some constructions and ideas of
\cite{Tiago-Claudio-Marco-CodimensaoFoldFold} and
\cite{Jeffrey-colombo} are very useful in our approach.

Now we establish a convention on the trajectories of orbit-solutions
of a PSVF. 

\begin{definition}\label{definicao trajetorias}
The \textbf{local trajectory (orbit)} $\phi_{Z}(t,p)$ of a PSVF
given by \eqref{eq Z} through $p\in V$ is defined as follows:
\begin{itemize}
\item For $p \in \Sigma^+ \backslash \Sigma = \{q\in V \, | \, z >  0\}$ and $p \in \Sigma^{-} \backslash \Sigma = \{q\in V \, | \, z< 0\}$ the trajectory is
given by $\phi_{Z}(t,p)=\phi_{X}(t,p)$ and
$\phi_{Z}(t,p)=\phi_{Y}(t,p)$ respectively.

\item For $p \in \Sigma^{c+}$ and  taking the
origin of time at $p$, the trajectory is defined as
$\phi_{Z}(t,p)=\phi_{Y}(t,p)$ for $t \leq 0$ and
$\phi_{Z}(t,p)=\phi_{X}(t,p)$ for $t \geq 0$. For
the case $p \in \Sigma^{c-}$  the definition is the same
reversing time.

\item For $p \in \Sigma^e$ and  taking the
origin of time at $p$, the trajectory is defined as
$\phi_{Z}(t,p)=\phi_{Z^{\Sigma}}(t,p)$ for $t \leq 0$ and
$\phi_{Z}(t,p)$ is either $\phi_{X}(t,p)$ or $\phi_{Y}(t,p)$ or $\phi_{Z^{\Sigma}}(t,p)$ for $t \geq 0$. For
$p \in \Sigma^s$ the definition is the same
reversing time.

\item For $p$ a regular tangency point and  taking the
origin of time at $p$, the trajectory is defined as
$\phi_{Z}(t,p)=\phi_{1}(t,p)$ for $t \leq 0$ and
$\phi_{Z}(t,p)=\phi_{2}(t,p)$ for $t \geq 0$, where each $\phi_{1},\phi_{2}$ is either $\phi_{X}$ or $\phi_{Y}$ or $\phi_{Z^{\Sigma}}$.

\item For $p$ a singular tangency point
    $\phi_{Z}(t,p)=p$ for all $t \in \R$.
\end{itemize}
\end{definition}

\begin{definition}\label{definicao trajetoria global}
The \textbf{orbit} (\textbf{trajectory}) of a point $p \in V$ is the
set $\gamma (p) = \{ \phi_{Z}(t,p) : t \in \R \}$ obtained by the
concatenation of local trajectories.
\end{definition}

Consider $0 \neq p\in \Sigma^{c+}$. It is easy to see that there
exists a time $t_1(p) > 0$, called $X$-fly time, such that the
forward trajectory of $X$ passing through $p$ at $t=0$ return to
$\Sigma$ after $t_1(p)$. We define the \textit{half return map
associated to $X$} by  $\varphi_{X}(p)=\phi_X(t_1(p),p)=p_1\in
\Sigma$. When $p_1 \in \Sigma^{c-}$, let $t_2(p_1)>0$ be the $Y$-fly
time of the trajectory of $Y$ passing through $p_1$.  Define the
\textit{half return map associated to $Y$} by
$\varphi_{Y}(p_1)=\phi_Y(t_2(p_1),p_1) \in \Sigma$. The $C^r$
involution $\varphi_X$ (resp. $\varphi_Y$) is such that
$Fix(\varphi_X) = S_X$ (resp. $Fix(\varphi_Y) = S_Y$). The {\it
first return map} associated to $Z=(X,Y)$ is defined by the
composition of these involutions, i.e.,
\begin{equation}\label{eq primeiro retorno geral}
\varphi_Z(p)= \varphi_Y \circ \varphi_X (p) = \phi_Y^{}(t_2(p_1),\phi_X^{}(t_1(p),p))
\end{equation}or the reverse, applying first the flow of  $Y$ and after the flow of $X$. See Figure \ref{fig primeiro retorno} and details in \cite{Marco-enciclopedia}.

\begin{figure}[!h]
\begin{center}\psfrag{A}{$p$}\psfrag{B}{$p_1$}\psfrag{C}{$\varphi_Z(p)$} \epsfxsize=6cm
\epsfbox{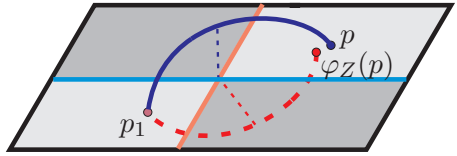} \caption{\small{First Return Map.}}
\label{fig primeiro retorno}
\end{center}
\end{figure}

The mapping  $\varphi_Z$ is an important object in order  to study
the behavior of $Z$ around a T-singularity. The proofs of the main
results require a detailed analysis of the return map and different
domains (departure regions), according to the Filippov's
decomposition of $\Sigma$, must be considered. For example,  if $q
\in \Sigma^{e}$ then the iteration of both mappings $\varphi_Y \circ
\varphi_X (q)$ and $\varphi_X \circ \varphi_Y (q)$ must be
considered.

\section{Properties of System $Z_0$ given by \eqref{eq Z sem F}}\label{secao propriedades sistema}

In this section we describe some important features about the PSVF given by \eqref{eq Z sem F}.
 In Subsection \ref{secao campo deslizante campo inicial} we describe the sliding vector field associated to \eqref{eq Z sem F}.
 In Subsection \ref{secao aplicacao retorno campo inicial} we analyze the first return map associated to it.
 In Subsection \ref{secao convergencia trajetorias campo inicial} we analyze the way in which the trajectories converge to a limit set.
  This last analysis permits us to detect the behavior of certain invariant sets.

\subsection{The sliding vector field associated to \eqref{eq Z sem F}}\label{secao campo deslizante campo inicial}

Using Equation \eqref{equacao campo normalizado}, associated to \eqref{eq Z sem F} we have the normalized sliding vector field
\begin{equation}\label{equacao campo normalizado Z sem F}
Z_{0}^{s}(x,y,z)=((-1-(x+y)) x  + y,(1-(x+y))x -y,0).
\end{equation}

Let us understand the phase portrait of \eqref{equacao campo normalizado Z sem F}.

\begin{proposition}\label{proposicao sing campo normalizado}
The normalized sliding vector field $Z^{s}_{0}$, given by \eqref{equacao campo normalizado Z sem F}, has a saddle-node at the origin.
\end{proposition}
\begin{proof} Identify $\Sigma$ with the $xy$-plane.
Consider the change of variables $(u,v)=(x+y,x-y)$. Then
\eqref{equacao campo normalizado Z sem F} can be re-written in the
form $(\dot{u},\dot{v}) = (-(u+v) u , - 2 v)$. This last system has
the origin as a unique equilibrium with eigenvectors $v_1=(1,0)$,
$v_2 = (0,1)$ associated to the eigenvalues $\lambda_1 = 0$,
$\lambda_2=-2$ respectively.  The phase portrait is pictured at
Figure \ref{fig campo deslizante normalizado}.

\begin{figure}[!h]
\begin{center}\psfrag{X}{$\Sigma$} \epsfxsize=4cm
\epsfbox{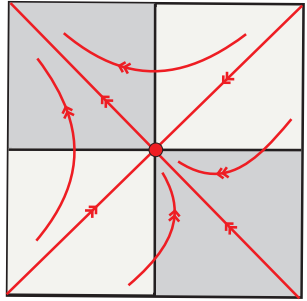} \caption{\small{Phase Portrait of the normalized sliding vector field $Z^{s}_{0}$.}}
\label{fig campo deslizante normalizado}
\end{center}
\end{figure}

\end{proof}

\begin{remark}\label{remark sela hiperbolica}
Since \eqref{equacao campo normalizado Z sem F} has a saddle-node at
the origin, by Remark \ref{obs campo normalizado topol equivalente}
we must reverse the orientation in $\Sigma^e$. So, we conclude that
the sliding vector field associated to $Z_0$ has the phase portrait
shown at Figure \ref{fig campo deslizante sem normalizar}. Note that
the straight line $y+x=0$ in $\Sigma$, where $Y_3 - X_3 =0$ in
\eqref{eq campo filippov}, is composed only by equilibrium points of
the sliding vector field associated to $Z_0$. Moreover, except by
the stable invariant manifold, the trajectories of the sliding
vector field associated to $Z_0$ depart from $\Sigma^e$.
\begin{figure}[!h]
\begin{center}\psfrag{X}{$\Sigma$} \epsfxsize=4cm
\epsfbox{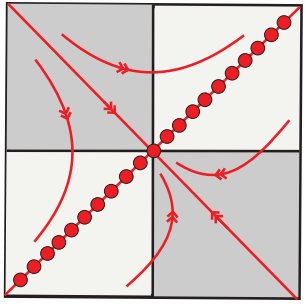} \caption{\small{Phase Portrait of the sliding vector field associated to $Z_0$.}}
\label{fig campo deslizante sem normalizar}
\end{center}
\end{figure}
\end{remark}

\subsection{The first return map associated to \eqref{eq Z sem F}}\label{secao aplicacao retorno campo inicial}

In order to exhibit the first return map associated to
$Z_0=(X,Y_0)$, given by \eqref{eq Z sem F},  we have to write the
expressions of the half return maps $\varphi_X$ and
$\varphi_{Y_{0}}$. A straightforward calculation shows that the
trajectories $\phi_X$ and $\phi_{Y_{0}}$ are parameterized,
respectively, by
\begin{equation}\label{eq trajetorias X}
\phi_{X}(t) = (-t + k_1 e^{-2t} + k_2 , t + k_1 e^{-2t} - k_2 , -t^2/2 + k_1 e^{-2t}/2 + k_2 t + k_3 ),
\end{equation}and
\begin{equation}\label{eq trajetorias Y}
\phi_{Y_{0}}(t) = (t + l_1, -t + l_2, t^2/2 + l_1 t + l_{1}^{2}/2 +  l_3 ),
\end{equation}

Consider the planes $\pi_k = \{ (x,y,z) \in V \, | \, y=-x+k  \}$,
with $k\in \R$.

\begin{proposition}\label{proposicao retorno X} Given an arbitrary point $(x_{0},y_{0},0) \in \Sigma^{c+}$ then \linebreak $\varphi_{X}(x_{0},y_0,0) = (-t_2 + ((x_0+y_0)/2) e^{-2t_{2}} + (x_0 - y_0)/2 , t_2 + ((x_0+y_0)/2) e^{-2t_{2}} - (x_0 - y_0)/2 ,0)$, where the $X$-fly time $t_2>0$ is given implicitly by \begin{equation}\label{eq raiz nao soluvel}
-t_{2}^{2}/2 + ((x_0+y_0)/4) e^{-2t_{2}} + ((x_0 - y_0)/2) t_{2} -(x_0 + y_0)/4=0.
\end{equation}
In particular,  $\phi_{X}(\pi_0 \cap \Sigma^+) \subset (\pi_0 \cap \Sigma^+)$ and $\varphi_{X}(x_{0},-x_0,0) = (-x_0 , x_0 ,0).$
\end{proposition}
\begin{proof} Considering the initial condition $(x_{0},y_{0},0)$ and \eqref{eq trajetorias X}, let $t_2 >0$ be the first time such that \[-t^{2}/2 + ((x_0+y_0)/4) e^{-2 t} + ((x_0 - y_0)/2) t -(x_0 + y_0)/4=0.\]
Using \eqref{eq trajetorias X} it is easy to see that
\[\varphi_{X}(x_{0},y_0,0) =\] \[(-t_2 + ((x_0+y_0)/2) e^{-2 t_{2}} + (x_0
- y_0)/2 ,  t_2 + ((x_0+y_0)/2) e^{-2t_{2}} - (x_0 - y_0)/2 ,0).\]
In particular, for $p_0=(x_{0},y_{0},0) = (x_0, -x_0,0) \in \pi_0
\cap \Sigma$, we obtain that the first two coordinates of
$\phi_{X}(t,p_0)$ are $x(t) = -t  +  x_0 $ and $y(t) = t  -  x_0 =
-x(t)$. Moreover, in this case we can solve explicitly the equation
$-t_{2}^{2}/2 + ((x_0+y_0)/2) e^{-2t_{2}}/2 + ((x_0 - y_0)/2) t_{2}
-(x_0 + y_0)/4=0$ and we obtain $t_2=2 x_0$. So,
 $\varphi_{X}(x_{0},-x_0,0) = (-x_0 , x_0 ,0)$. This concludes the proof.
\end{proof}

\begin{proposition}\label{proposicao retorno Y} Given an arbitrary point $(x_{0},y_{0},0) \in \Sigma^{c-}$ then \linebreak $\varphi_{Y_{0}}(x_{0},y_0,0) = (-x_0,  y_0 + 2 x_0,0)$. In particular,  $\phi_{Y_{0}}(\pi_k \cap \Sigma^-) \subset (\pi_k \cap \Sigma^-)$.
\end{proposition}
\begin{proof}Considering the initial condition $(x_{0},y_{0},0)$ and \eqref{eq trajetorias Y}, in order to determine $\varphi_{Y_{0}}$ it is enough to obtain the first time  $t_1 >0$ such that $t_{1}^{2}/2 + x_0 t_1 =0$. So $t_1 = -2 x_0$ and $\varphi_{Y_{0}}(x_{0},y_0,0) = (-x_0,  y_0 + 2 x_0,0)$. In particular, for $p_0=(x_{0},y_{0},0) = (x_0, -x_0 + k,0) \in \pi_k \cap \Sigma^{-}$, the first two coordinates of $\phi_{Y}(t,p_0)$ are $x(t)=t + x_0$ and $y(t)=-t -x_0 + k=-(t + x_0)+k=-x(t)+k$. This concludes the proof.
\end{proof}

\begin{proposition}\label{proposicao plano pi zero invariante}
The plane $\pi_0 = \{ (x,y,z) \in V \, | \, x+y=0 \}$ is
$Z_0$-invariant. Moreover, $Z_{0}|_{\pi_{0}}$ is a  center.
\end{proposition}
\begin{proof}
By Propositions \ref{proposicao retorno X} and \ref{proposicao retorno Y}  we get that $\pi_0$ is invariant by the flow of  $Z_0$. In order to see that $Z_0$ has a center at $\pi_0$ it is enough to see that $$\varphi_{Z_{0}}(x_0,-x_0,0)=\varphi_{Y_{0}}\circ\varphi_{X}(x_0,-x_0,0)=\varphi_{Y_{0}}(-x_0 , x_0 ,0)=(x_0,-x_0,0).$$
\end{proof}


Now we will prove that $\pi_0$ is a hyperbolic global attractor for the trajectories of $Z_{0}$.

\begin{proposition}\label{proposicao convergencia para pi zero}
Let $r_0$ be the straight line given by $r_0= \pi_0 \cap \Sigma$. Given $(x_{0},y_0,0) \in \Sigma^{c}$ then $$d(\varphi_{Z_{0}}(x_{0},y_0,0),r_0) < d((x_{0},y_0,0),r_0).$$This means that the trajectories of $Z_0$ are converging to $r_0$. Moreover, $$\varphi_{Z_{0}}^{n}(x_{0},y_0,0)=(-x_n, x_n + (x_0 + y_0) e^{-2 (t_{2}^{(1)}+ \hdots + t_{2}^{(n)})},0)$$where $t_{2}^{(i)}$ is the fly time necessary to the $X$-trajectory by $\varphi_{Z_{0}}^{i}(x_0,y_0,0)$ returns to $\Sigma$ and $x_n = - t_{2}^{(n)} + ((x_0+y_0)/2) e^{-2 (t_{2}^{(1)}+ \hdots + t_{2}^{(n)})} - x_{n-1} - (x_0 + y_0) (e^{-2 (t_{2}^{(1)}+ \hdots + t_{2}^{(n-1)})})/2$.
\end{proposition}
\begin{proof}By Proposition \ref{proposicao retorno X} we obtain
\[\varphi_{X}(x_{0},y_0,0) = \]\[(-t_2 + ((x_0+y_0)/2) e^{-2 t_{2}} + (x_0
- y_0)/2 , t_2 + ((x_0+y_0)/2) e^{-2 t_{2}} - (x_0 - y_0)/2
,0),\]where $t_2>0$ is given implicitly by \[-t_{2}^{2}/2 +
((x_0+y_0)/2) e^{-2 t_{2}}/2 + ((x_0 - y_0)/2) t_{2} -(x_0 +
y_0)/4=0.\]

By Proposition \ref{proposicao retorno Y},
\begin{equation}\label{equacao primeiro retorno}
\begin{array}{cccc}
  \varphi_{Z_{0}}(x_{0},y_0,0)  =  (\varphi_{Y} \circ \varphi_{X})(x_{0},y_0,0) = \\
    =  \Big(t_2 - (\frac{(x_0+y_0)}{2} e^{-2t_{2}} - \frac{(x_0 - y_0)}{2},-t_2 + 3(\frac{(x_0+y_0)}{2} e^{-2t_{2}} + \frac{(x_0 - y_0)}{2},0\Big).&
\end{array}
\end{equation}

Then we get,
$$
d(\varphi_{Z_{0}}(x_{0},y_0,0),r_0) = \frac{\sqrt{2}}{2} (x_0+y_0) e^{-2 t_{2}}  < \frac{\sqrt{2}}{2} (x_0+y_0) = d((x_{0},y_0,0),r_0).
$$

In order to obtain that \[\varphi_{Z_{0}}^{n}(x_{0},y_0,0)=(-x_n,
x_n + (x_0 + y_0) e^{-2 (t_{2}^{(1)}+ \hdots + t_{2}^{(n)})},0),\]
with \[x_n = - t_{2}^{(n)} + \frac{(x_0+y_0)}{2} e^{-2 (t_{2}^{(1)}+
\hdots + t_{2}^{(n)})} - x_{n-1} - \frac{(x_0 + y_0) (e^{-2
(t_{2}^{(1)}+ \hdots + t_{2}^{(n-1)})})}{2},\] it is enough to use
$n$ times Propositions \ref{proposicao retorno X} and
\ref{proposicao retorno Y}.
\end{proof}

%


\begin{proposition}\label{proposicao autovalores primeiro retorno}
Let $p_0=(x_0,y_0,0) \in \pi_0$. The first return map $\varphi_{Z}(p_0)$ has eigenvalues $\mu_1=1$ and $\mu_2= 1- 4 x_0 + 8 x_{0}^{2} + O(x_{0}^{2})$ and eigenvectors $u_1=(-1,1)$ and $u_2=(-(1-2 x_0 + 2 x_0)/(1 - 4 x_0 + 6 x_{0}^{2}) + O(x_{0}^{2}),1)$, respectively.
\end{proposition}
\begin{proof}
Since we can not solve Equation \eqref{eq raiz nao soluvel} we consider the second order series expansion of $e^{-2t}$. So we are able to solve \eqref{eq raiz nao soluvel} and obtain the $X$-fly time of $p_0$. Moreover, a straightforward calculation shows that the diffeomorphism $\varphi_{Z}(p_0)$ has eigenvalues  $\mu_1=1$ and $\mu_2= 1- 4 x_0 + 8 x_{0}^{2}$ and eigenvectors $u_1=(-1,1)$ and $u_2=(-(1-2 x_0 + 2 x_0)/(1 - 4 x_0 + 6 x_{0}^{2}),1)$, respectively.
\end{proof}

\begin{proposition}\label{proposicao pi zero eh atrator para Z rho}
All trajectories of $Z_0$, given by \eqref{eq Z sem F}, converge to the plane $\pi_0$.
\end{proposition}
\begin{proof} First note that each $q \in V$ hits $\Sigma$ for some positive time.
Using Remark \ref{remark sela hiperbolica}, the position of the eigenspaces of \eqref{equacao campo normalizado Z sem F} and  Propositions \ref{proposicao convergencia para pi zero} and \ref{proposicao autovalores primeiro retorno} it is easy to see that given $p \in \Sigma = \Sigma^{s} \cup \Sigma^{d} \cup \Sigma^{c}$ the trajectories of $Z_0$  by $p$ converge to $\pi_0$.
\end{proof}

\subsection{The convergence of the trajectories}\label{secao convergencia trajetorias campo inicial}

Now we picture the scenario describing the asymptotic behavior of
$Z_0$ in $V$. As we said above, the plane $\pi_0$ is an attractor
for $\phi_{Z}(t,p)$ and separates $V$ in two open regions
$\mathcal{V^{+}} = \{ (x,y,z) \in V \, | \, x+y>0\}$ and
$\mathcal{V^{-}} = \{ (x,y,z) \in V \, | \, x+y<0\}$.

Since all points of $V$ hit $\Sigma$, in order to determine the
asymptotic behavior of a trajectory it is enough to take this
trajectory departing just from points in $\Sigma$. Consider the
partition of $\Sigma = \Sigma^{c+}_{s} \cup \Sigma^{c+}_{e} \cup
\Sigma^{e} \cup \Sigma^{c-}_{e} \cup \Sigma^{c-}_{s} \cup \Sigma^s
\cup   r_{0}^{+} \cup r_{0}^{-} \cup S_{Y}^{-} \cup S_{X}^{-}\cup
S_{Y}^{+}\cup S_{X}^{+} \cup 0 \}$, where $\Sigma^{c+}_{s} =
\Sigma^{c+} \cap V^+$, $\Sigma^{c+}_{e} = \Sigma^{c+} \cap V^-$,
$\Sigma^{c-}_{e} = \Sigma^{c-} \cap V^-$, $\Sigma^{c-}_{s} =
\Sigma^{c-} \cap V^+$,  $r_{0}^{+}=r_0 \cap \Sigma^{c+}$,
$r_{0}^{-}=r_0 \cap \Sigma^{c-}$, $S_{Y}^{-}=S_{Y} \cap V^{-}$,
$S_{X}^{-}=S_{X} \cap V^{-}$, $S_{Y}^{+}=S_{Y} \cap V^{+}$,
$S_{X}^{+}=S_{X} \cap V^{+}$ and $0=\{(0,0,0)\}$. See Figure
\ref{fig divisao convergencia}.

\begin{figure}[!h]
\begin{center}\psfrag{X}{$\Sigma$} \epsfxsize=8cm
\psfrag{A}{$\Sigma^{c+}_{s}$} \psfrag{B}{$\Sigma^{c+}_{e}$}\psfrag{C}{$\Sigma^e$}\psfrag{D}{$\Sigma^{c-}_{e}$}\psfrag{E}{$\Sigma^{c-}_{s}$}\psfrag{F}{$\Sigma^s$}\psfrag{G}{$r_{0}^{+}$}\psfrag{}{$$}\psfrag{}{$$}\psfrag{}{$$}\psfrag{}{$$}
\psfrag{H}{$S_{Y}^{-}$}\psfrag{I}{$S_{X}^{-}$}\psfrag{J}{$r_{0}^{-}$}\psfrag{K}{$S_{Y}^{+}$}\psfrag{L}{$S_{X}^{+}$}\epsfbox{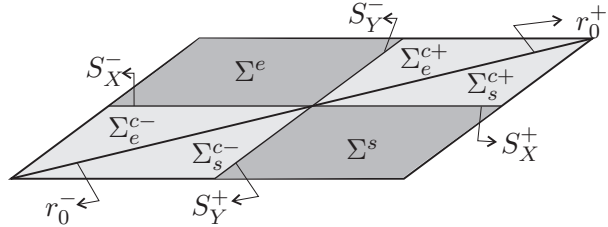} \caption{\small{Partition of $\Sigma$.}}
\label{fig divisao convergencia}
\end{center}
\end{figure}

Consider the following steps:

\begin{itemize}
\item[\textbf{(1)}]First let us analyze the trajectories in $r_0$.

\subitem\textbf{(1.i)} If $p=0$ then, according to the fifth
bullet of Definition \ref{definicao trajetorias},
$\phi_{Z}(t,p)=0$ for all $t \in \R$.

\subitem\textbf{(1.ii)} If $p \in r_{0}^{-} \cup r_{0}^{+}$
then, according to Proposition \ref{proposicao plano pi zero
invariante},  $\phi_{Z}(p)$ became restrict to $\pi_0$ and
describes a periodic orbit around the origin.

\item[\textbf{(2)}]Let us analyze the trajectories in $V^{+} \cap \Sigma$.

\subitem\textbf{(2.i)} If $p \in \Sigma^{s}$ then, according to
Proposition \ref{proposicao sing campo normalizado},
$\phi_{Z}(t,p)\rightarrow 0$ when $t \rightarrow + \infty$. By
the fourth bullet of Definition \ref{definicao trajetorias}, the
same holds when $p \in S^{+}_{Y} \cup S_{X}^{+}$.

\subitem\textbf{(2.ii)} If $p = (x_p,y_p,0) \in \Sigma^{c+}_{s}$
then we can put $y_p = - M x_p$, where $0<M<1$. With this, we
are able to explicitly calculate the successive returns
$\varphi_{Z}^{n}(p)=(\Pi_{1}^{n}(p),\Pi_{2}^{n}(p),0)$, with $n
\geq 1$. A straightforward calculation shows that
$\Pi_{2}^{n}(0)$ is positive when $2n/(2n + 1) > M$. This
implies that after the first integer $N$ such that $2N/(2N + 1)
> M$ we obtain $\Pi_{2}^{n}(p)>0$ and so $\Pi^{n}(p)$ belongs to
$\Sigma^{s}$, where $y>0$. By the analysis done in (2.i),
$\phi_{Z}(t,p)\rightarrow 0$ when $t \rightarrow + \infty$. This
means that the competition between the two attractors $\pi_0$
and $\Sigma^s$ is won by $\Sigma^s$.

\subitem\textbf{(2.iii)} If $p = (x_p,y_p,0) \in
\Sigma^{c-}_{s}$ then $\phi_{Z}(p,-2 x_p) \in \Sigma^{c+}_{s}$
and so, repeat the analysis of item (2.ii).

\item[\textbf{(3)}]Let us analyze the trajectories in $V^{-} \cap \Sigma$.

\subitem\textbf{(3.i)} If $p \in \Sigma^{c+}_{e} \cup S^{-}_{Y}$
then we can put $y_p = - R x_p$, where $0<R<1$. With this, we
are able to explicitly calculate the successive returns
$\varphi_{Z}^{n}(p)=(\Pi_{1}^{n}(p),\Pi_{2}^{n}(p),0)$, with $n
\geq 1$. A straightforward calculation shows that
$\varphi_{Z}^{n}(p)\rightarrow \infty$. So,
$\phi_{Z}(t,p)\rightarrow \infty$ when $t \rightarrow + \infty$.

\subitem\textbf{(3.ii)} If $p  \in \Sigma^{c-}_{e} \cup
S_{X}^{-}$ then $\phi_{Z}(p,-2 x_p) \in \Sigma^{c+}_{e}$ and so,
repeat the analysis of the previous bullet.

\subitem\textbf{(3.iii)} If $p \in \Sigma^{e}$ then, according
to the third bullet of Definition \ref{definicao trajetorias},
there are three choices for  $\phi_{Z}(t,p)$. When
$\phi_{Z}(t,p)=\phi_{X}(t,p)$ or when
$\phi_{Z}(p,t)=\phi_{Y}(t,p)$ the trajectory hits
$\Sigma^{c-}_{e}$ or $\Sigma^{c+}_{e}$, respectively. In both
cases we use the previous two items. When
$\phi_{Z}(t)=\phi_{Z^{s}}(t,p)$ then, by Remark \ref{remark sela
hiperbolica}, there exists an invariant manifold in $\Sigma^{e}$
converging to the origin and the other points in $\Sigma^{e}$
converge to $S^{-}_{Y} \cup S^{-}_{X}$ (and so, we use the
previous two items).  This means that the competition between
the attractor $\pi_0$ and the repulsiveness of $Z$ in $p \in
\Sigma^{c+}_{e}$ is won by the second one.

\end{itemize}

\begin{remark}\label{obs convergencia trajetorias}
As consequence of the previous analysis we are able to say that any
$Z$-trajectory, with $Z$ given by \eqref{eq Z sem F}, through $p \in
V$ either describes a periodic orbit or converges to a stationary
point or leaves any ball around the origin, according to the initial
position of $p$. 
\end{remark}

\section{Properties of an oriented perturbation of System $Z_{0}$}\label{secao propriedades sistema perturbado}

\subsection{Auxiliary results}\label{secao resultados auxiliares}

In what follows, $h : \R \rightarrow \R$ will denote the $C^\infty$-function given by
\[h(w)=\left\{\begin{array}{ll}
0, & \mbox{ if }w\leq0;\\
e^{-1/w}, &\mbox{ if }w>0.
\end{array}\right.\]

Consider the function $F_\varepsilon^{\rho}(x,y): \R^{2} \times \R$, where either $\rho=f$ or $\rho=i$, such that
\begin{equation}\label{eq F finito}F_\varepsilon^{f}(x,y)= - \varepsilon  h(x) h(-y)
(\varepsilon-x)(2\varepsilon-x)\dots(k\varepsilon-x)\end{equation}with $k \in \N$, and \begin{equation}\label{eq F infinito}F_\varepsilon^{i}(x,y)=  h(x) h(-y) \sin(\pi \varepsilon^2 /x).\end{equation}

\begin{lemma}\label{lema finitos pontos fixos trajetorias}
Consider the  function $F_\varepsilon^{f}(x,y)$ given by \eqref{eq F finito}.
\begin{enumerate}
\item[(i)] If $\varepsilon<0$ then $F_\varepsilon^f$ does not have roots in $(0,+\infty) \times \{ y \}$.
\item[(ii)] If $\varepsilon>0$ then $F_\varepsilon^f$
has exactly $k$ roots in $(0,+\infty) \times \{ y \}$, these roots are
$\{(\varepsilon,y),(2\varepsilon,y),\dots,(k\varepsilon,y)\}$.
\item[(iii)] $\displaystyle\frac{\partial F_\varepsilon^f}{\partial
x}(j\varepsilon,y)= - \epsilon h(-y)  (-1)^j \varepsilon^k
h(j\varepsilon)(k-j)! (j-1)! \mbox{ for }j\in\{1,2,\dots,k\}$.
It means that such partial derivative  at $(j\varepsilon,y)$ is
positive for $j$ odd and negative for $j$ even.
\end{enumerate}
\end{lemma}
\begin{proof}
When $x>0$, by a straightforward calculation $F_\varepsilon^f(x,y)=0$
if, and only if, $(\varepsilon-x)(2\varepsilon-x)\dots(k\varepsilon-x)=0$.
So, the roots of $F_\varepsilon^f(x,y)$ in $(0,+\infty)$  are
$\varepsilon,2\varepsilon,\dots,k\varepsilon$. Moreover,
$$\displaystyle\frac{\partial F_\varepsilon^f}{\partial x}(x,y)= -\epsilon h(-y)
 \displaystyle\frac{\partial }{\partial x}\Big((j \varepsilon - x) H(x) \Big) = -\epsilon h(-y)
 \Big( (j \varepsilon - x) \displaystyle\frac{\partial H}{\partial x}(x) -  H(x) \Big),$$ where $H(x) = F_\varepsilon^f(x,y)/( \epsilon h(-y) (j \varepsilon -
 x))$. So,
$$\begin{array}{l}
  \displaystyle\frac{\partial F_\varepsilon^f}{\partial x}(j \varepsilon,y)  =   \epsilon h(-y) H(j \varepsilon) =  \\
   = \epsilon h(-y) \varepsilon^{k} h(j \varepsilon) (1-j)\dots((j-1)-j )((j+1)-j)\dots(k-j) \\
    = - \epsilon h(-y) \varepsilon^{k} h(j \varepsilon) (-1)^{j} \Big((j-1)\dots(j- (j-1))\Big) \Big(((j+1)-j)\dots(k-j)\Big)  \\
     =  - \epsilon h(-y) (-1)^j\varepsilon^kh(j\varepsilon)(k-j)! (j-1)!  \\
 \end{array}$$
This proves items (ii) and (iii). Item (i) follows immediately.
\end{proof}

\begin{lemma}\label{lema infinitos pontos fixos trajetorias}
Consider the  function $F_\varepsilon^{i}(x,y)$ given by \eqref{eq F infinito}.  For $\varepsilon\neq0$ the function $F_\varepsilon^i$ has infinitely many
roots in $(0,\varepsilon^2)\times \{ y \}$, these roots are $\{(\varepsilon^2,y),(\varepsilon^2/2,y),(\varepsilon^2/3,y),\dots\}$ and
$$\displaystyle\frac{\partial F_\varepsilon^i}{\partial
x}(\varepsilon^2/j,y)= h(-y) (-1)^j  (-\pi j^2/\varepsilon^2)h(\varepsilon^2/j) \mbox{ for
}j\in\{1,2,3,\dots\}.$$ It means that such derivative at  $(\varepsilon^2/j,y)$
is positive for $j$ odd and negative for $j$ even.
\end{lemma}
\begin{proof}
When $x>0$, by a straightforward calculation $F_\varepsilon^i(x,y)=0$
if, and only if, $\sin(\pi \varepsilon^2/x)=0$. So, the roots of
$F_\varepsilon^i(x,y)$ in $(0,\varepsilon^2)\times \{ y \}$  are $(\varepsilon^2,y),(\varepsilon^2/2,y),(\varepsilon^2/3,y),\dots$. Moreover,
$$\displaystyle\frac{\partial F_\varepsilon^i}{\partial x}(x,y)=
h(-y) [ h^{\prime}(x) \sin(\pi \varepsilon^2/x)  + h(x)\cos(\pi \varepsilon^2/x)(-\pi\varepsilon^2/x^2)].$$ So,
$$\begin{array}{rcl}
  \displaystyle\frac{\partial F_\varepsilon^i}{\partial x}(\varepsilon^2/j,y) & = & h(-y) [ h^{\prime}(\varepsilon^2/j) \sin(\pi \varepsilon^2/j) +\\
  & + &    h(\varepsilon^2/j)\cos(\pi \varepsilon^2/j)(-\pi j^2/\varepsilon^2)] \\
&    = & h(-y)(-1)^j (-\pi j^2/\varepsilon^2) h(\varepsilon^2/j). \\
 \end{array}$$
\end{proof}

Consider $Z_0$ given by \eqref{eq Z sem F} and

\begin{equation}\label{eq Z inicial}
 Z^{\rho}_{\varepsilon}(x,y,z) = \left\{
      \begin{array}{ll}  X(x,y,z)=(-1-(x+y),1-(x+y),-y)    & \hbox{if $z \geq 0$,} \\
        Y^{\rho}_{\varepsilon}(x,y,z) = (1,-1, x + \frac{\partial F^{\rho}_{\varepsilon}}{\partial x}(x,y))                                        & \hbox{if $z \leq 0$.}
      \end{array}    \right.
\end{equation}

\begin{remark}\label{remark convergencia dos campos}
Take $Z_{\varepsilon} = Z^{\rho}_{\varepsilon}$, with $\rho=i,f$. It is easy to see that $Z_{\varepsilon} \rightarrow Z_0$ when $\varepsilon \rightarrow 0$.
\end{remark}

Associated to  \eqref{eq Z inicial} we have the normalized sliding vector field given by
\[
Z^{s}_{\rho,\varepsilon}(x,y,z)=\]\[\Big((-1-(x+y))( x +
\frac{\partial F^{\rho}_{\varepsilon}}{\partial x}(x,y)) +
y,(1-(x+y))(x + \frac{\partial F^{\rho}_{\varepsilon}}{\partial
x}(x,y)) -y,0\Big).
\]

A straightforward calculation shows that the trajectory
$\phi_{Y^{\rho}_{\varepsilon}}$ of $Y^{\rho}_{\varepsilon}$ given in
\eqref{eq Z inicial} are parameterized by
\begin{equation}\label{eq trajetorias Y rho}
\phi_{Y^{\rho}_{\varepsilon}}(t) = (t + l_1, -t + l_2, t^2/2 + l_1 t + F^{\rho}_{\varepsilon}(t + l_1,-t + l_2) + l_3 ).
\end{equation}

\begin{proposition}\label{proposicao retorno Y rho} Given an arbitrary point $(x_{0},y_{0},0) \in \Sigma^{c-}$ then \linebreak $\varphi_{Y^{\rho}_{\varepsilon}}(x_{0},y_0,0) = (t_1+x_0, -t_1 + y_0,0)$, where the $Y^{\rho}_{\varepsilon}$-fly time $t_1>0$ is given implicitly by $t_{1}^{2}/2 + x_0 t_1 + F^{\rho}_{\varepsilon}(t_1+x_0,-t_1+y_0) =0$. In particular,  $\phi_{Y^{\rho}_{\varepsilon}}(\pi_k \cap \Sigma^-) \subset (\pi_k \cap \Sigma^-)$.
\end{proposition}
\begin{proof} Let $t_1 >0$ the first time such that $t_{1}^{2}/2 + x_0 t_1 + F^{\rho}_{\varepsilon}(t_1+x_0,-t_1+y_0)=0$.
Using \eqref{eq trajetorias Y rho} it is easy to see that $\varphi_{Y^{\rho}_{\varepsilon}}(x_{0},y_0,0) = (t_1+x_0, -t_1 + y_0,0)$. In particular, for $p_0=(x_{0},y_{0},0) = (x_0, -x_0 + k,0) \in \pi_k \cap \Sigma^{-}$, the first two coordinates of $\phi_{Y^{\rho}_{\varepsilon}}(t,p_0)$ are $x(t)=t + x_0$ and $y(t)=-t -x_0 + k=-(t + x_0)+k=-x(t)+k$. This concludes the proof.
\end{proof}

\begin{proposition}\label{proposicao plano pi zero invariante perturbado}
The plane $\pi_0$ is invariant by the flow of $Z^{\rho}_{\varepsilon}$, given by \eqref{eq Z inicial}.
\end{proposition}
\begin{proof}
By Proposition \ref{proposicao retorno Y rho} and Proposition \ref{proposicao retorno X}  we get that $\pi_0$ is invariant by the flow of $Z^{\rho}_{\varepsilon}$.
\end{proof}

\begin{remark}\label{remark plano invariante rho}
Since, by Proposition \ref{proposicao convergencia para pi zero}, we get that all trajectories of $Z_0$ converge to $\pi_0$ we obtain that, for $\varepsilon$ sufficiently small, the same holds for the trajectories of $Z^{\rho}_{\varepsilon}$.
\end{remark}


%

\begin{proposition}\label{proposicao perturbacao k ciclos} Consider an integer $k \geq 0$ and  $Z^{\rho}_{\varepsilon}$ given by \eqref{eq Z inicial}, where either $\rho=f$ or $\rho=i$. Then
  $Z^{f}_{\varepsilon}$ has exactly $k$ limit cycles and  $Z^{i}_{\varepsilon}$ has infinite many limit cycles, all of then situated in $\pi_0$.
\end{proposition}
\begin{proof}
According to Remark \ref{remark plano invariante rho}, $\pi_0$ is a global attractor for $Z^{\rho}_{\varepsilon}$. Also, according to Proposition \ref{proposicao plano pi zero invariante perturbado}, $\pi_0$ is invariant by the flow of $Z^{\rho}_{\varepsilon}$. So, if there exists limit cycles, then they are situated at $\pi_0$. Moreover when we restrict the flow of $Z^{\rho}_{\varepsilon}$ to $\pi_0$, by Propositions \ref{proposicao retorno Y rho} and \ref{proposicao retorno X}, the fixed points of the first return map $\varphi_{Z^{\rho}_{\varepsilon}} = \varphi_{Y^{\rho}_{\varepsilon}} \circ \varphi_{X}$ occurs when $t=t_3=4 x_0$. So, take $p_0 = (x_{0},-x_0,0)$ and we get
\begin{equation}\label{equacao aplicacao retorno em x=-y}
\varphi_{Z^{\rho}_{\varepsilon}}(p_0) = \phi_{Y^{\rho}_{\varepsilon}}(2 x_0, -p_0) =  (x_0, - x_0, F^{\rho}_{\varepsilon}(x_0,- x_0)).
\end{equation}
When $\rho=f$, by Item (ii) of Lemma \ref{lema finitos pontos fixos trajetorias}, $$\varphi_{Z^{\rho}_{\varepsilon}}(x_{0},-x_0,0)=(x_0, - x_0,0) \Leftrightarrow x_0 = j \varepsilon \mbox{ with } \varepsilon>0 \mbox{ and } j=1,2, \hdots, k.$$Therefore, $Z^{f}_{\varepsilon}$ has $k$ limit cycles, all of then situated in $\pi_0$.
When $\rho=i$, by  Lemma \ref{lema infinitos pontos fixos trajetorias}, $$\varphi_{Z^{\rho}_{\varepsilon}}(x_{0},-x_0,0)=(x_0, - x_0,0) \Leftrightarrow x_0 = \varepsilon^{2}/j \mbox{ with }  j=1,2, \hdots.$$Therefore, $Z^{i}_{\varepsilon}$ has infinite many limit cycles, all of then situated in $\pi_0$.
\end{proof}

%

\begin{proposition}\label{proposicao ciclos sao hiperbolicos}
All limit cycles in Proposition \ref{proposicao perturbacao k
ciclos} are hyperbolic. Moreover, for $\epsilon>0$, if
$j=even$ then it is attractor and if $j=odd$ then it is repeller.
\end{proposition}
\begin{proof}
In fact, in order to prove this we must consider the expression of
the derivatives in Item (iii) of Lemma \ref{lema finitos pontos
fixos trajetorias} and Lemma \ref{lema infinitos pontos fixos
trajetorias}. Observe that when $k=0$ in Proposition \ref{proposicao perturbacao k
ciclos} there is not limit cycles. In this Item (iii) of Lemma \ref{lema finitos pontos
fixos trajetorias} also is true and the origin is an attractor equilibrium of the system.
\end{proof}

\begin{remark}\label{obs origem}
	Observe that when $k=0$ in Proposition \ref{proposicao perturbacao k
		ciclos} the system has not limit cycles. In this Item (iii) of Lemma \ref{lema finitos pontos
		fixos trajetorias} also is true and the origin is an attractor equilibrium of the system.
\end{remark}

\subsection{About the convergence of the trajectories of $Z^{\rho}_{\varepsilon}$, given by \eqref{eq Z inicial}}\label{secao convergencia perturbado}

Now we will proceed the analysis of the behavior of the trajectories of \eqref{eq Z inicial} in a neighborhood of the T-singularity (at the origin).

As stated in Propositions \ref{proposicao plano pi zero invariante
perturbado} and Remark \ref{remark plano invariante rho}, $\pi_0$ is
invariant by the flow of \eqref{eq Z inicial} and all trajectories
of \eqref{eq Z inicial} in a neighborhood of the T-singularity
converge to it. However, after the perturbation imposed to $Z_0$,
the asymptotic behavior of the trajectories can drastically changes.
The biggest change occurs in $V^{-}$ when $\rho=f$ and $\epsilon
>0$. Following Proposition \ref{proposicao ciclos sao hiperbolicos},
we can separate the analysis of this situation, essentially, in
three cases:

$\bullet$ When $k=0$ we get that $Z^{\rho}_{\varepsilon}$ does not
presents limit cycles. So,  the trajectories in $V^{-} \backslash
\Sigma^{e}$ becomes increasingly distant from the origin until
certain moment. However, since the origin is an attractor when we
restrict the analysis to $\pi_0$, there exists a moment from which
the trajectory starts to converge to origin. So, the T-singularity
at the origin of \eqref{eq Z inicial} is asymptotically stable.

$\bullet$ When $k\neq 0$ is even  we get that
$Z^{\rho}_{\varepsilon}$ presents $k$ hyperbolic \textbf{nested}
limit cycles and $\Gamma_{k}$, the bigger of them, is attractor. So,
the trajectories in $V^{-} \backslash \Sigma^{e}$ becomes
increasingly distant from the origin until certain moment. However,
since the bigger limit cycle is an attractor when we restrict the
analysis to $\pi_0$, there exists a moment from which the trajectory
starts to converge to it. The same holds for all attractor limit
cycles at the interior of $\Gamma_{k}$. So, the T-singularity at the
origin of \eqref{eq Z inicial} is asymptotically stable. Also, each
one of the repeller hyperbolic limit cycles at the interior of
$\Gamma_k$ is repeller and each one of the attractor hyperbolic
limit cycles is an attractor for the trajectories in $V^{-} \cup
\pi_0$. In fact, there are topological half-cylinder of orbits
converging to each it one of them.

$\bullet$ When $k\neq 0$ is odd we get that $Z^{\rho}_{\varepsilon}$
presents $k$ hyperbolic \textbf{nested} limit cycles and
$\Gamma_{k}$, the bigger of them, is repeller. So, there are
trajectories in $V^{-} \backslash \Sigma^{e}$ leaving any
neighborhood of the origin. Also, there are trajectories in $V^{-}
\backslash \Sigma^{e}$ converging to the origin or, when $k\geq 3$,
to  an inward hyperbolic limit cycle. Moreover, the T-singularity at
the origin of \eqref{eq Z inicial} is asymptotically stable. Each
one of the repeller hyperbolic limit cycles is a repeller and each
one of the attractor hyperbolic limit cycles is an attractor for the
trajectories in $V^{-}\cup \pi_0$. In fact, there are topological
half-cylinder of orbits  converging to each it one of them.

\section{Proof of main results}\label{secao prova resultados}

Now we prove the main results of the paper:

\begin{proof}[\textbf{Proof of Theorem A:}] Item (a): It follows from Remark \ref{remark convergencia dos campos}.

Item (b): It follows from Propositions \ref{proposicao perturbacao k
ciclos} and \ref{proposicao ciclos sao hiperbolicos}.

Item (c): It follows from Propositions \ref{proposicao pi zero eh atrator para Z rho} and Remark \ref{remark plano invariante rho}.

Item (d): It follows from Remark \ref{obs origem}.

\end{proof}

Before to prove Theorem B, let us define the classical notion of codimension of vector fields.

%

\begin{definition}
Consider $\Theta(W) \subset \Omega^r$ a small neighborhood a vector field $W$. We say that $W$ has codimension $k$ if it appears $k$ distinct topological types of vector fields in $\Theta(W)$.
\end{definition}

\begin{proof}[\textbf{Proof of Theorem B:}] Suppose that the codimension of the
 T-singularity in \ref{eq Z sem F} is $m < \infty$. Then, in a neighborhood of $Z_0$ there are PSVFs of, at most, $m$ distinct topological types. This is a contradiction due Theorem A. So, the codimension of this singularity is infinite.
\end{proof}

\section{Conclusion}\label{secao conclusão}

As usual, the main aim of the perturbation theory is to approximate
a given dynamical system by a more familiar one, regarding the
former as a perturbation of the latter. The problem is to deduce
dynamical properties of the \textit{unperturbed} from the
\textit{perturbed} case. In this sense, we focus on certain PSVFs
$Z^{\rho}_{\varepsilon}$ which are deformations of $Z_0$.

In what follows we present a rough description of some properties
that  $Z_0=(X,Y_0)$ and
$Z^{\rho}_{\varepsilon}=(X,Y^{\rho}_{\varepsilon})$ enjoy
simultaneously. Some of them are:

\begin{itemize}
\item The origin is an equilibrium point of both $Z_0=(X,Y_0)$ and $Z^{\rho}_{\varepsilon}=(X,Y^{\rho}_{\varepsilon})$. Moreover, the tangency sets $S_{X}$ and $S_{Y_{0}}$ (resp. $S_{Y^{\rho}_{\varepsilon}}$), meet transversally at the origin,

\item The  plane $\pi_0 = \{y+x=0\}$ is $Z_0$ and $Z^{\rho}_{\varepsilon}$-invariant.


\item  The sliding vector fields associated to $Z_0$ and $Z^{\rho}_{\varepsilon}$ are topologically equivalent.
\end{itemize}

Also, there are properties that $Z_0$ and $Z^{\rho}_{\varepsilon}$
do not enjoy simultaneously. For example:

\begin{itemize}
\item The T-singularity at the origin of $Z^{\rho}_{\varepsilon}$ is asymptotically stable.


    \item The PSVF $Z_0$ restricted to $\pi_0$ has a center while in $Z^{\rho}_{\varepsilon}$ this center is perturbed giving rise to hyperbolic limit cycles at $\pi_0$.

\end{itemize}

%

Finally, as a conclusion of this paper we are able to say that the T-singularity of the PSVF \eqref{eq Z sem F} has infinite codimension.

\section{Appendix}\label{secao apendice}

Now we illustrate the theoretical analysis performed by means of some numerical simulations. In the next illustrations we use the computer program entitled "Mathematica".

We use the following line of commands:

\[solutions =
 Table[First[
   NDSolve[\{ x'[t] == If[z[t] > 0,\]\[ -1 - (x[t] + y[t]), 1],
     y'[t] == If[z[t] > 0, 1 - (x[t] + y[t]), -1],\]\[
     z'[t] == If[z[t] > 0, -y[t], x[t]], x[0] == \theta,
     y[0] == -\theta, \]\[z[0] == 0 \}, \{x, y, z\}, \{t, 0,
     2.5\}]], \{\theta, 0.1, 2 \pi - 0.1, 0.1\}];\]

     \[solutions2 =
 Table[First[
   NDSolve[\{ x'[t] == If[z[t] > 0,\]\[ -1 - (x[t] + y[t]), 1],
     y'[t] == If[z[t] > 0, 1 - (x[t] + y[t]), -1],\]\[
     z'[t] == If[z[t] > 0, -y[t], x[t]], x[0] == \cos[\theta]/4,
     y[0] == \sin[\theta]/4,\]\[ z[0] == 0 \}, \{x, y, z\}, \{t, 0,
     2.5\}]], \{\theta, 0.1, 2 \pi - 0.1, 0.1\}];\]

\[c1 = ContourPlot3D[z, \{x, -.4, .4\}, \{y, -.4, .4\}, \{z, -.5,
.5\},\]\[
   Contours \rightarrow 0, Mesh \rightarrow False];\]

   \[c2 = ContourPlot3D[y, \{x, -.5, .5\}, \{y, -.4, .4\}, \{z, -.5, .5\},\]\[
   Contours \rightarrow 0, Mesh \rightarrow False];\]

   \[c3 = ContourPlot3D[x+y, \{x, -.5, .5\}, \{y, -.4, .4\}, \{z, -.5, .5\},\]\[
   Contours \rightarrow 0, Mesh \rightarrow False];\]

 The command \[p3 = ParametricPlot3D[
  Evaluate[\{x[t], y[t], z[t]\} /. solutions], \{t, 0, 2.5\}, \]\[
  PlotStyle \rightarrow \{Thickness[.0015], Red\}]; Show[p3, ImageSize \rightarrow Large]\]generates Figure \ref{fig mathematica centro}.

\begin{figure}[!h]
\begin{center}\psfrag{X}{$\Sigma$} \epsfxsize=6cm
\epsfbox{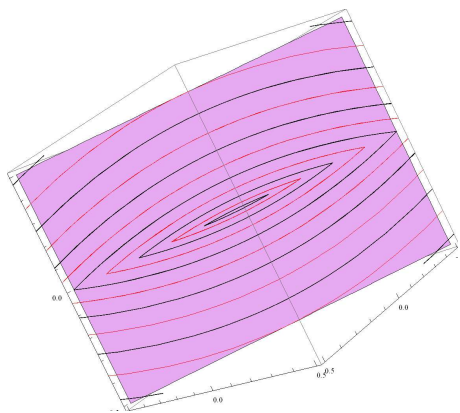} \caption{\small{Phase Portrait of the center of $Z_0$.}}
\label{fig mathematica centro}
\end{center}
\end{figure}

 The command \[p4 = ParametricPlot3D[
  Evaluate[\{x[t], y[t], z[t]\} /. solutions2], \{t, 0, 2.5\},
  \]\[
  PlotStyle \rightarrow \{Thickness[.0015], Red\}]; Show[p4, ImageSize \rightarrow Large]\]generates Figure \ref{fig mathematica trajetorias global}.

\begin{figure}[!h]
\begin{center}\psfrag{X}{$\Sigma$} \epsfxsize=7cm
\epsfbox{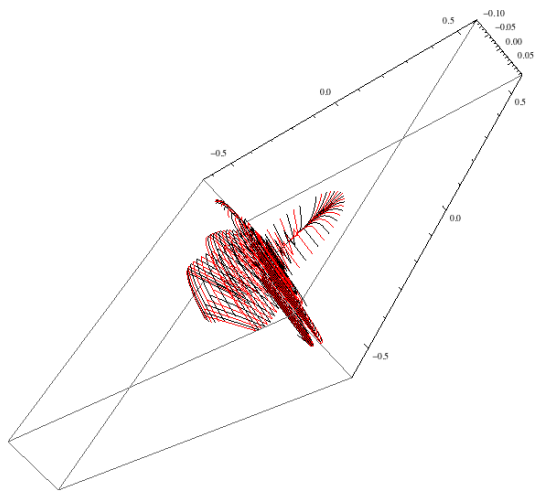} 
\end{center}\caption{\small{Phase Portrait of trajectories in a neighborhood of the origin.}}
	\label{fig mathematica trajetorias global}
\end{figure}

%


\vspace{.7cm}

\noindent {\textbf{Acknowledgments.} The first author was partially supported by \linebreak grant\#2014/02134-7, S\~{a}o Paulo Research Foundation (FAPESP), CAPES-Brazil grant number  88881.030454/2013-01 from the program CSF-PVE and a CNPq-Brazil grant number 443302/2014-6. The second author is
partially supported by grant \#2012/18780-0, S\~{a}o Paulo Research Foundation (FAPESP).



\begin{thebibliography}{9}



\bibitem{Anosov-stability} {\sc D.V. Anosov}, {\it Stability of the equilibrium positions in relay systems}, Automat. Teleme-kh., 20 (1959), pp. 135-149. E



\bibitem{diBernardo-electrical-systems} {\sc M. di Bernardo, A. Colombo and E. Fossas}, {\it Two-fold singularity in nonsmooth electrical systems}, Proc. IEEE International Symposium on Circuits ans Systems, (2011), 2713--2716.

\bibitem{Jeffrey-T-sing} {\sc M. di Bernardo, A. Colombo,E. Fossas and M.R. Jeffrey}, {\it Teixeira singularities in 3D switched feedback control systems}, Systems and Control Letters, vol. \textbf{59}, (2010), 615--622.


\bibitem{Brogliato} {\sc B. Brogliato}, {\it Nonsmooth Mechanics: Models, Dynamics and Control}, Springer- Verlag, New York, 1999.





\bibitem{Tiago-Claudio-Marco-CodimensaoFoldFold} {\sc C.A. Buzzi, T. de Carvalho and M.A. Teixeira},
{\it Birth of limit cycles bifurcating from a nonsmooth center}, Journal de Mathematiques Pures et Appliquees, volume 102, p. 36 - 47, 2014.

%
%





\bibitem{Jeffrey-colombo} {\sc A. Colombo and M.R. Jeffrey}, {\it The two-fold singularity of discontinuous vector fields}, SIAM J. Appl. Dyn. Syst., vol. \textbf{8}, (2009), 624--640.

\bibitem{Jeffrey-colombo-2011} {\sc A. Colombo and M.R. Jeffrey}, {\it Non-deterministic chaos, and the two fold singularity in piecewise smooth flows}, SIAM J. Appl. Dyn. Syst., vol. \textbf{10}, (2011), 423--451.

%


\bibitem{Rony}{\sc  R. Cristiano},
{\it Bifurca\,c\~oes em sistemas din\^amicos DPWS com aplica\,c\~oes em eletr\^onica de pot\^encia}, Master Thesis, UFSC, link http://tede.ufsc.br/teses/PEAS0125-D.pdf (2013) (in Portuguese).


%
%


\bibitem{Rossa} {\sc F. Dercole and F.D. Rossa}, {\it Generic and Generalized Boundary Operating Points in Piecewise-Linear (discontinuous) Control Systems}, In 51st IEEE Conference on Decision and Control, 10-13 Dec. 2012, Maui, HI, USA. Pages:7714-7719. DOI: 10.1109/CDC.2012.6425950.


\bibitem{Dixon} {\sc D.D. Dixon}, {\it Piecewise Deterministic Dynamics from the Application of Noise to Singular Equation of Motion}, J. Phys A: Math. Gen. 28, (1995), 5539-5551.


\bibitem{Fi} {\sc A.F. Filippov}, {\it Differential Equations with Discontinuous Righthand Sides}, Mathematics and its Applications (Soviet Series), Kluwer Academic
Publishers-Dordrecht, 1988.

%


\bibitem{J-T-T1} {\sc A. Jacquemard, M.A. Teixeira and D.J. Tonon}, {\it Stability conditions in piecewise smooth dynamical systems at a two-fold singularity}, Journal of Dynamical and Control Systems, vol. \textbf{19}, (2013), 47-67.

\bibitem{J-T-T2} {\sc A. Jacquemard, M.A. Teixeira and D.J. Tonon}, {\it Piecewise smooth reversible dynamical systems at a two-fold singularity}, International Journal of Bifurcation and Chaos, vol. \textbf{22}, (2012).

\bibitem{Kousaka} {\sc T. Kousaka, T.  Kido, T.  Ueta, H. Kawakami and M. Abe}, {\it Analysis of Border- Collision Bifurcation in a Simple Circuit}, Proceedings of the International Symposium on Circuits and Systems, II-481-II-484, (2000).



\bibitem{Kunze-1998}{\sc M. Kunze}, {\it Unbounded solutions in non-smooth dynamical systems at resonance}, Z. Angew Math Mech, vol. \textbf{78}, Supplement 3, (1998), 985-986.


\bibitem{Kupper-2008}{\sc T. K\"{u}pper}, {\it Invariant cones for non-smooth systems}, Mathematics and
Computers in Simulation, vol. \textbf{79}, (2008), 1396-1409.


\bibitem{Kupper-Hosham-2011} {\sc T. K\"{u}pper and H.A. Hosham}, {\it Reduction to invariant cones for non-smooth
systems}, special Issue of Mathematics and Computers in Simulation, vol. \textbf{81},
(2011), 980-995.

\bibitem{Leine} {\sc R. Leine and H.  Nijmeijer}, {\it Dynamics and Bifurcations of Non-Smooth Mechanical Systems}, Lecture Notes in Applied and Computational Mechanics, vol. 18, Berlin Heidelberg New-York, Springer-Verlag, (2004).


%
%
%
%
%

\bibitem{Teixeira-JDE-1977} {\sc M.A. Teixeira},
{\it  Generic Bifurcation in Manifolds with Boundary}, Journal of
Differential Equations \textbf{25} (1977), 65--88.

\bibitem{T1} {\sc M.A.Teixeira}, {\it Stability conditions for discontinuous vector fields}, J. Differential Equations, vol. \textbf{88} (1990), 15-29.


\bibitem{Marco-enciclopedia} {\sc M.A. Teixeira},
{\it  Perturbation Theory for Non-smooth Systems}, Meyers:
Encyclopedia of Complexity and Systems Science \textbf{152}, 2008.


%

\end{thebibliography}
\end{document}